\documentclass[12pt]{article}
\usepackage{geometry}
\usepackage{graphicx,color,mathtools}
\usepackage{amssymb,amsmath,amsthm,mathrsfs}
\usepackage[all,cmtip]{xy}
\usepackage{epstopdf, comment, url}
\usepackage{bm}
\usepackage{enumitem}

\usepackage[pdftex,bookmarks,pdfnewwindow,plainpages=false,unicode,pdfencoding=auto]{hyperref}

\newtheorem{theorem}{Theorem}[section]
\newtheorem{corollary}[theorem]{Corollary}
\newtheorem{lemma}[theorem]{Lemma}

\theoremstyle{remark}
\newtheorem*{remark}{Remark}

\numberwithin{equation}{section}

\DeclareMathOperator{\crit}{crit}

\DeclareMathOperator{\diam}{diam}

\DeclareMathOperator{\aut}{Aut}

\DeclareMathOperator{\loc}{loc}

\DeclareMathOperator{\GCE}{GCE}

\title{Analytic mappings of the unit disk with \\ bounded compression}
\author{Oleg Ivrii and Artur Nicolau}
\date{July 20, 2025}

\begin{document}
 
\maketitle

\begin{abstract}
In this paper, we study analytic self-maps of the unit disk for which the hyperbolic diameters of the images of hyperbolic balls of radius 1 are uniformly bounded below. We give several characterizations of such maps involving the behaviour along geodesic rays, Aleksandrov-Clark measures, zero sets and critical sets.
\end{abstract}

\section{Introduction}

Let $\mathbb{D} = \{ z \in \mathbb{C}: |z| < 1 \}$ be the unit disk in the complex plane and 
$$\lambda_{\mathbb{D}}(z) |dz| = \frac{2 |dz|}{1-|z|^2}
$$
be the hyperbolic metric on $\mathbb{D}$. We denote the  hyperbolic distance between two points $z, w \in \mathbb{D}$ by $d_h (z, w)$ and the hyperbolic ball centered at $z \in \mathbb{D}$ of radius $R$ by $B_h (z, R) = \{w \in \mathbb{D}: d_h (w,z) < R\}$.

Let $F: \mathbb{D} \rightarrow \mathbb{D}$ be an analytic self-map of the unit disk. The pullback of the hyperbolic metric under $F$ is $\lambda_{F}(z) |dz| = 2 |dz| \cdot |F'(z)| / (1-|F(z)|^2)$, $z \in \mathbb{D}$. The Schwarz lemma says that any analytic self-map $F$ of the unit disk contracts the hyperbolic metric, that is, $\lambda_F (z) \leq \lambda_{\mathbb{D}} (z)$ for any $z \in \mathbb{D}$. Moreover, equality at a single point implies equality at every point in the unit disk and that $F$ is an automorphism of $\mathbb{D}$.

We say that an analytic self-map of the unit disk has {\em Bounded Compression} (BC) if the hyperbolic diameters of the images of hyperbolic balls of radius 1 are uniformly bounded below. In this paper, we give a number of equivalent characterizations of BC mappings involving the behaviour along geodesic rays, Aleksandrov-Clark measures, zero sets and critical sets.

A discrete set of points in the unit disk is {\em separated} if the hyperbolic distance between any two points in the set is bounded below by a positive constant. More generally, a set is {\em quasi-separated} if there exists an $N$ so that every ball of hyperbolic radius 1 contains at most $N$ elements of the set. It is not difficult to see that a set is quasi-separated if and only if it is a finite union of separated sets. 

For an arc  $I \subset \partial \mathbb{D}$ of length $|I| < 1/2$, we write
$$
Q_I = \{ z \in \mathbb{D} : z/|z| \in I, \, 1 - |I| < |z| < 1 \}
$$
for the Carleson square with base $I$. Given a Carleson square $Q=Q_I$, we denote its ``side length'' by $\ell (Q) = |I|$. We denote the midpoint of $I$ by $\xi_I \in \partial \mathbb{D}$ and the ``center'' of the Carleson square by $z_Q = z_I = (1-|I|/2) \xi_I \in \mathbb{D}$.

Let $\mu$ be a measure on the unit circle. We say that an arc $I \subset \partial \mathbb{D}$ is {\em heavy} for $\mu$ if
$$
\frac{\mu(I)}{|I|} \ge \frac{1}{100} \cdot u(z_I),
$$
where $u$ is the Poisson extension of $\mu$. Note that the converse inequality with 100 replaced by a different absolute constant always holds.

A measure $\sigma$ on the unit disk is a {\em Carleson measure} if the collection of numbers $\{ \sigma(Q) / \ell(Q) \}$ as $Q$ ranges over Carleson squares is bounded. We say that a Carleson square $Q$ is {\em heavy} for $\sigma$ if
$$
\frac{\sigma(Q)}{\ell(Q)} \ge \frac{1}{100} \int_{\mathbb{D}} \frac{1-|z_Q|^2}{|1-\overline{z_Q} w|^2} \, d\sigma(w).
$$
Note that the converse inequality with 100 replaced by a different absolute constant always holds.

\begin{theorem}
\label{mainth}
Let $F$ be an analytic self-map of the unit disk. The following conditions are equivalent:

\begin{enumerate}[label={\normalfont(\arabic*)}, ref=\arabic*]
\item \label{item:hyperbolic-diameter}
Bounded compression: There exists a constant $c_1 > 0$ such that for any $z \in \mathbb{D}$,
\begin{equation*}
\label{eq:diameter-condition}
\diam_h \bigl (F(B_h(z,1)) \bigr ) \ge c_1.
\end{equation*}

\item \label{item:hyperbolic-descent} Hyperbolic descent: For every $0 < \varepsilon < 1$, there exists an integer $N = N(\varepsilon) > 0$ with the property that for any Carleson square $Q \subset \mathbb{D}$, there exists a Carleson sub-square $Q' \subset Q$ with $\ell(Q') = 2^{-N} \ell(Q)$ such that
$$
1 - |F(z_{Q'})| \le \varepsilon ( 1 - |F(z_Q)|).
$$

\item \label{item:quasigeodesic-condition} Quasigeodesic condition: There exist constants $C > 0$ and $0 < s \le 1$ such that for any $z \in \mathbb{D}$, there exists a geodesic ray $\gamma = \gamma(z)$ emanating from $z$ with
\begin{equation}
\label{eq:quasi-geodesic}
d_h(F(w_1), F(w_2)) \ge s \, d_h(w_1, w_2) - C, \qquad w_1, w_2 \in \gamma.
\end{equation}

\item \label{item:ac-measures} Aleksandrov-Clark measures: Let $\mu = \mu_\alpha$ be one of the Aleksandrov-Clark measures of $F$. For every $\varepsilon > 0$, there exists a $\delta = \delta(\varepsilon) > 0$ such that every heavy arc $I \subset \partial \mathbb{D}$ has a sub-arc $J \subset I$ with 
$$
|J| \ge \delta |I| \qquad \text{and} \qquad \frac{\mu(J)}{|J|} \le \varepsilon \cdot \frac{\mu(I)}{|I|}.
$$

\item \label{item:zero-sets} Zero sets: 
\begin{enumerate}
\item[{\em (a)}] $F$ is a Blaschke product.
\item[{\em (b)}] $\sigma = \sum_{F(z) = 0} (1-|z|^2) \delta_{z}$ is a Carleson measure.
\item[{\em (c)}] For every $\varepsilon > 0$, there exists a $\delta = \delta(\varepsilon) > 0$ such that for any heavy
Carleson square $Q$, there exists $Q' \subset Q$ with $\ell(Q') \ge \delta \, \ell(Q)$ such that
$$
\frac{\sigma(Q')}{\ell(Q')} \le \varepsilon \cdot \frac{\sigma(Q)}{\ell(Q)}.
$$
\end{enumerate}

\item \label{item:critical-points} Critical points: $F$ is a maximal Blaschke product whose  critical set $C = \crit F$ is quasi-separated and has uniform upper density $D^+(C) < 1$.
\end{enumerate}
\end{theorem}

We now clarify the terms involved in the theorem above and present a number of additional remarks:

\begin{enumerate}[wide]

\item {\bf Aleksandrov-Clark measures.}  Given an analytic mapping $F$ from the unit disk to itself and a point $\alpha \in \partial \mathbb{D}$, the function $(\alpha + F)/(\alpha - F)$ has positive real part. Consequently, there exists a positive measure $\mu_\alpha = \mu_\alpha (F)$ on the unit circle and a constant $C_{\alpha} \in \mathbb{R}$ such that
\begin{equation}
\label{eq:ACMeasureHerglotz}
\frac{\alpha + F(z)}{\alpha - F(z)} = \int_{\partial \mathbb{D}} \frac{\xi + z}{\xi - z}\, d\mu_\alpha(\xi) + iC_{\alpha}, \qquad z \in  \mathbb{D}. 
\end{equation}
The measures $\{\mu_\alpha\colon \alpha \in \partial \mathbb{D}\}$ are called the {\em Aleksandrov-Clark} measures of the function $F$. These measures were introduced by D.~Clark in relation with operator theory and were thoroughly  investigated by A.~B.~Aleksandrov who recognized their importance in function theory.
We refer the reader to the surveys  \cite{PolSar, Saks} as well as \cite[Chapter~IX]{cima} for further properties of Aleksandrov-Clark measures and a wide range of applications.

\item {\bf Maximal Blaschke products.} If $C$ is the critical set of some analytic self-map of the unit disk, then it is the critical set of a particular Blaschke product $F = F_C$ called the {\em maximal Blaschke product} which maximizes $\lambda_F(z)$, $z \in \mathbb{D} \setminus C$, out of all analytic self-maps of the unit disk with critical set $C$. The maximal Blaschke product $F_C$  is uniquely determined up to post-composition with a M\"obius transformation. For general properties of maximal Blaschke products, we refer the reader to \cite{kraus, KR-survey, KR-maximal}.

\item {\bf Uniform upper density.} Following K.~Seip \cite{seip}, the {\em uniform upper density} of a discrete set $C \subset \mathbb{D}$ is given by
$$
D^+(C) = \limsup_{r \to 1} \, \sup_{a \in \mathbb{D}} \, D(m_{a \to 0}(C), r),
$$
where
$$
D(C, r) = \frac{ \sum_{1/2 < |c_j| < r} (1 - |c_j|)}{\log \frac{1}{1-r}}
\qquad \text{and} \qquad
m_{a \to 0}(z) = \frac{z-a}{1-\overline{a}z}.
$$
From the definition, it is clear that the density $D^+(C)$ is conformally invariant. Since the density $D^+$ is stable under hyperbolic perturbations, from Condition (\ref{item:critical-points}), we can read off the following result:

\begin{corollary}
Suppose $F_C$ is a BC mapping with critical set $C = \{c_j\}$. There is a $\delta > 0$ depending on $F_C$ such that for any discrete set of points  $C' = \{c'_j\}$ in the unit disk with $d_h(c_j, c'_j) < \delta$, the maximal Blaschke product $F_{C'}$ with critical set $C'$ exists and is BC.
\end{corollary}

\item {\bf Almost isometries.} In \cite{GP91}, J.~Garnett and M.~Papadimitriakis investigated {\em almost isometries}\/, defined as analytic self-mappings $F$ of the unit disk for which there exists a constant $c > 0$ such that
$$
\diam_h(F(B_h(z,R))) \ge 2R - c,
$$
for all $z \in \mathbb{D}$ and $R > 0$. They showed that $F$ is an almost isometry if and only if there exists a constant $C_1 > 0$ such that every point $z \in \mathbb{D}$ lies within hyperbolic distance $C_1$ of a two-sided geodesic $\gamma = \gamma(z)$ such that
$$
d_h(F(w_1), F(w_2)) \ge d_h(w_1, w_2) - C_1, \qquad w_1, w_2 \in \gamma.
$$
This is just the two-sided version of the condition (\ref{eq:quasi-geodesic}) with $s=1$.

\item {\bf APHA mappings.} The hyperbolic area of a measurable set $E \subset \mathbb{D}$ is 
$$
A_h (E) = \int_E \frac{4\, dA(z)}{(1-|z|^2)^2 } . 
$$
In \cite{APHA}, the authors investigated analytic mappings of the unit disk which {\em almost preserve hyperbolic area} (or APHA, for short) defined by the condition that there exists a constant $c > 0$ such that
\begin{equation*}
\label{eq:APHA-condition}
A_h \bigl (F(B_h(z,R)) \bigr ) \ge c \, A_h (B_h(z,R)), \qquad z \in \mathbb{D}, \qquad R \ge 1.
\end{equation*}
Evidently, any APHA mapping has bounded compression, but the converse is far from true. Since BC mappings are maximal Blaschke products, from Condition (\ref{item:critical-points}), it follows that a BC mapping is APHA if and only if its critical points form a Blaschke sequence that satisfies the Carleson condition
$$
\sum_{c_j \in Q:\, F'(c_j) = 0} (1-|c_j|) < c \cdot \ell(Q),
$$
for every Carleson square $Q$ in the unit disk.

\item {\bf H\"older behaviour.} From our analysis of Conditions (\ref{item:hyperbolic-descent}) and (\ref{item:quasigeodesic-condition}) in Section \ref{sec:geodesic-rays}, it follows that mappings with bounded compression exhibit H\"older behaviour at almost every point of the unit circle:

\begin{corollary}
\label{holder-behaviour}
Let $F$ be an analytic mapping of the unit disk with bounded compression. There is an $0 < s \le 1$ such that
$$
\limsup_{r \to 1} \frac{1 - |F(r\xi)|}{(1-r)^s} < \infty, \qquad m \text{ a.e.~} \xi \in \partial \mathbb{D}.
$$
\end{corollary}

\item {\bf Porosity.} A closed set $E \subset \partial \mathbb{D}$ is called {\em porous} if there exists a constant $\delta >0$ such that every arc $I \subset \partial \mathbb{D}$ contains a sub-arc $J$ with $|J| \ge \delta |I|$ and $J \cap E = \varnothing$. From Condition (\ref{item:ac-measures}), it is immediate that:

\begin{corollary}
Any positive measure supported on a porous subset of the unit circle is the Aleksandrov-Clark measure of a BC mapping.
\end{corollary}

Likewise Condition (\ref{item:zero-sets}) leads to the following result:

\begin{corollary}
Let $F$ be a Blaschke product with zeros $\{z_n\}$. Assume that $\sigma = \sum (1-|z_n|^2) \delta_{z_n}$ is a Carleson measure. Assume also that there exists a constant $\delta >0$ such that any Carleson square $Q$ contains a Carleson sub-square $Q' \subset Q$, with $\ell(Q') \ge \delta \, \ell(Q)$, such that $Q'$ has no zeros of $F$. Then, $F$ is a BC mapping.
\end{corollary}

\item {\bf Equivalence of conformal metrics.} The proof of the equivalence (\ref{item:hyperbolic-diameter}) $\Leftrightarrow$ (\ref{item:critical-points}) relies on the following theorem, which may be of independent interest:

\begin{theorem}
\label{bounded-gap-theorem}
Suppose $I$ is an inner function and $J$ is an analytic self-map of the unit disk. If
$$
\frac{|I'(z)|}{1 - |I(z)|^2} \asymp \frac{|J'(z)|}{1 - |J(z)|^2}, \qquad z \in \mathbb{D},
$$
then there exists an automorphism of the unit disk $\tau$ such that $I = \tau \circ J$.
\end{theorem}

The proof will be presented in Section \ref{sec:maximality-criterion}.

\item {\bf Zero sets of Bergman spaces and interpolating sequences.} For $1 \le p < \infty$ and $\alpha > -1$, the {\em weighted Bergman space} $A^p_\alpha$ consists of analytic functions on the unit disk
for which
$$
\int_{\mathbb{D}} |f(z)|^p (1-|z|)^\alpha \, dA(z) < \infty.
$$
In the beautiful work \cite{kraus}, D.~Kraus showed that $C$ is the critical set of some Blaschke product if and only if it is the zero set of an $A^2_1$ function. Unfortunately, zero sets of functions in weighted Bergman spaces do not have a simple characterization. Nevertheless, as shown in \cite[Theorem 4.37]{HKZ}, zero sets of functions in
$$
A^{p-}_\alpha = \bigcap_{q<p} A^q_\alpha
$$
are described by the condition $D^+_\kappa(A) \le (1+\alpha)/p$. The density $D^+_\kappa$ is defined in terms of Privalov stars (as the definition is cumbersome, it will not be reproduced here). As explained in Seip's paper \cite{seip}, for separated sets, the density $D^+_\kappa$ is equivalent to the uniform upper density $D^+$. A.~Hartmann \cite{hartmann} observed that the same is true for quasi-separated sets. Consequently, any quasi-separated set with $D^+(A) < 1$ is the zero set of $A^2_1$ function, while any quasi-separated zero set necessarily satisfies $D^+(A) \le 1$.

A discrete set of points $C = \{c_j \}$ in the unit disk is called an {\em interpolating set} for $A^2_1$ if for every sequence $\{a_j \}$ such that $\{(1-|c_j|)^{-3/2} a_j \} \in \ell^2$, there exists a function $f \in A^2_1$ satisfying $f(c_j) = a_j$ for all $j$. Seip \cite{seip} showed that $C$ is an $A^2_1$ interpolating set if and only if it is separated and $D^+(C) < 1$. Shortly thereafter, Schuster and Seip \cite{SS98} studied a seemingly weaker notion of interpolation (whose definition involves Bergman canonical divisors) and showed that weakly interpolating sets were characterized by the same condition. Our proof of Characterization (\ref{item:critical-points}) can be viewed as an analogue of \cite{SS98} in the context of analytic self-maps of the unit disk. Somewhat surprisingly, the use of the Schwarz lemma simplifies the argument considerably in our setting.
\end{enumerate}

\section{Preliminary observations}

The hyperbolic derivative of an analytic self-map $F$ of the unit disk is given by
$$
D_h F(z) \, = \, \frac{\lambda_F}{\lambda_{\mathbb{D}}}(z) \, = \,
\frac{(1-|z|^2) \, |F'(z)|}{1-|F(z)|^2}, \quad z \in \mathbb{D}.
$$
The Schwarz lemma says that $D_h F(z) \le 1$ for any $z \in \mathbb{D}$ and equality holds at a single point if and only if $F$ is an automorphism of the unit disk.

In the introduction, we defined mappings of bounded compression using the hyperbolic diameter. As the following simple theorem shows, one could just as well have used the hyperbolic area, the hyperbolic area counted with multiplicity or asked that the image of every hyperbolic ball of radius 1 contain a hyperbolic ball of some fixed radius:

\begin{theorem}
\label{mainth2}
Let $F$ be an analytic self-map of the unit disk. The following conditions are equivalent:

\begin{enumerate}[label={\normalfont(1\alph*)}, ref=1\alph*, start=1]
\item \label{item:hyperbolic-diameter2}
Hyperbolic diameter: There exists a constant $c_1 > 0$ such that for any $z \in \mathbb{D}$,
\begin{equation*}
\label{eq:diameter-condition2}
\diam_h \bigl (F(B_h(z,1)) \bigr ) \ge c_1.
\end{equation*}

\item \label{item:hyperbolic-area}
Hyperbolic area: There exists a constant $c_2 > 0$ such that for any $z \in \mathbb{D}$,
\begin{equation*}
\label{eq:area-condition-without-multiplicity}
A_h \bigl (F(B_h(z,1)) \bigr ) \ge c_2.
\end{equation*}

\item \label{item:hyperbolic-area2} Hyperbolic area with multiplicity: There exists a constant $c_3 > 0$ such that for any $z \in \mathbb{D}$,
\begin{equation*}
\label{eq:area-condition-with-multiplicity}
\int_{B_{h}(z,1)} \frac{4 |F'(w)|^2}{(1-|F(w)|^2)^2} \, dA(w) \ge c_3.
\end{equation*}

\item \label{item:hyperbolic-derivative} Hyperbolic derivative: There exists a constant $c_4 > 0$ such that for any $z \in \mathbb{D}$,
\begin{equation*}
\label{eq:derivative-condition}
\max_{w \in B_h(z,1)} |D_h F(w)| > c_4.
\end{equation*}

\item \label{item:images-balls} Images of hyperbolic balls: There exists a constant $0 < c_5 < 1$ such that for any $z \in \mathbb{D}$,
\begin{equation*}
\label{eq:image-condition}
F(B_h(z,1)) \supset B_h(F(z), c_5).
\end{equation*}

\item \label{item:distortion-crit}
Distortion away from the critical set: The critical set of $F$ is quasi-separated and for every $\varepsilon > 0$, there exists a $\delta > 0$ such that if $d_h(z, \crit F) > \varepsilon$, then the hyperbolic derivative $D_h F(z) > \delta.$
\end{enumerate}
\end{theorem}

The equivalence of the above conditions comes from simple compactness arguments. We omit the proof. Note that in each characterization, the choice of hyperbolic radius 1 is not essential -- any fixed hyperbolic radius could be used.

\begin{lemma}
\label{island}
Let $F$ be an analytic self-map of the unit disk with bounded compression. Then $F$ is an inner function. In fact, $F$ is an indestructible Blaschke product, that is, $\tau \circ F$ is a Blaschke product for any M\"obius transformation $\tau \in \aut(\mathbb{D})$.
\end{lemma}

\begin{proof}
By \cite[Lemma 1.8]{mashreghi}, it is enough to show that if $F$ has a non-tangential limit at a point $\xi \in \partial \mathbb{D}$, then it cannot lie in the unit disk.
Indeed, if the non-tangential limit at $\xi \in \partial \mathbb{D}$ is $a \in \mathbb{D}$, then
the hyperbolic diameters of $F(B_h(r\xi, 1))$ tend to 0 as $r \to 1$, contradicting the assumption on $F$.
\end{proof}

\begin{lemma}
Let $F$ be an analytic self-map of the unit disk with bounded compression. Then there exists an $0 < a < 1$ such that for any $z_0 \in \mathbb{D}$ and $R \ge 1$, the image of $B(z_0, R)$ contains $B_h(F(z_0), a R)$.
\end{lemma}

\begin{proof}
It suffices to treat the case when $R$ is an integer. We will show that in this case, the lemma holds with $a = c_5$. (For general $R \ge 1$, one can take $a = c_5/2$). Let $w$ be a point in $B_h(F(z_0), c_5 R)$. We partition the geodesic segment $[F(z_0), w]$ into $R$ pieces of equal length by marking the points $w_0 = F(z_0), \dots, w_R = w$. Using the images of hyperbolic balls criterion, one can inductively choose the points $z_1, z_2, \dots, z_R$ with $F(z_{i+1}) = w_{i+1}$ and $d_h(z_{i+1}, z_i) < 1$ for $i=0, \dots, R-1$. By the triangle inequality, $z = z_R$ is a point with $F(z) = w$ whose hyperbolic distance to $z_0$ is less than $R$.
\end{proof}

\section{Geodesic rays and hyperbolic descent}
\label{sec:geodesic-rays}

In this section, we show the equivalence of Conditions (\ref{item:hyperbolic-derivative}), (\ref{item:hyperbolic-descent}) and (\ref{item:quasigeodesic-condition}).

\begin{proof}
(\ref{item:hyperbolic-derivative}) $\Rightarrow$ (\ref{item:hyperbolic-descent}). Fix a Carleson square 
$$
Q \, = \, Q(I) \, = \, \bigl \{r \xi \, : \, \xi \in I, \, 0 < 1-r < \ell(Q) \bigr \}.
$$
Given $\varepsilon > 0$, let $N = N(\varepsilon) > 0$ be a large integer to be determined later. Applying Green's identity to the functions $\log(1-|F(z)|^2)$ and $\log(1/|z|)$ in the domain
$$
Q_N = \bigl \{ r \xi \in Q \, : \,  2^{-N} \ell (Q) < 1 - r < \ell(Q) \bigr \},
$$
we get
\begin{equation}
\label{eq:green}
- \int_{\partial Q_N \cap \{ 1 -r = \ell(Q) \} } \log (1-|F|^2) \, \frac{ds}{|z|} + \int_{\partial Q_N \cap \{ 1 -r = 2^{-N} \ell(Q) \} } \log (1-|F|^2) \, \frac{ds}{|z|}
\end{equation}
\begin{equation*}
= - \int_{\partial Q_N} \partial_{\textbf{n}} \log(1-|F(z)|^2) \log \frac{1}{|z|} \, ds + \int_{Q_N} \Delta \log(1-|F(z)|^2) \log \frac{1}{|z|} \, dA(z).
\end{equation*}
As $\bigl (1-|z|^2 )|\nabla \log(1-|F(z)|^2) \bigr | \lesssim 1$,
$$
\int_{\partial Q_N} \partial_{\mathbf{n}} \log(1-|F(z)|^2) \log \frac{1}{|z|} \, ds = O(\ell(Q)).
$$
Furthermore, since
$$
\Delta \log(1-|F|^2) = - 4 \, \frac{|F'|^2}{(1-|F|^2)^2},
$$
the area integral
$$
\int_{Q_N} \Delta \log (1-|F(z)|^2) \log \frac{1}{|z|} dA(z) \, \asymp \,  - 4 \int_{Q_N} D_h^2 F(z) \, \frac{dA(z)}{1-|z|} \, \lesssim \, -N \ell(Q),
$$
by Condition (\ref{item:hyperbolic-derivative}). Hence, if $N$ is sufficiently large, then (\ref{eq:green}) reduces to
$$
- \int_{\partial Q_N \cap \{ 1 - r = \ell(Q) \} } \log (1-|F|^2) \, d\theta + \int_{\partial Q_N \cap \{ 1 - r = 2^{-N} \ell(Q) \} } \log (1-|F|^2) \, d\theta \, \lesssim \, -N \ell(Q).
$$
Consequently, there exists a Carleson square $Q' \subset Q$ with $\ell(Q') = 2^{-N} \ell(Q)$ such that
$$
- \log (1 - |F(z_Q)|^2) + \log(1 - |F(z_{Q'})|) \lesssim -N,
$$
that is, there exists an absolute constant $c>0$ such that $1-|F(z_{Q'})| \le 2^{-cN} (1-|F(z_Q)|)$.

\medskip

(\ref{item:hyperbolic-descent}) $\Rightarrow$ (\ref{item:quasigeodesic-condition}). Fix a point $z \in \mathbb{D}$ and pick a Carleson square $Q$ containing $z$ with $1-|z| \ge \ell(Q)/2$. Denote $Q_0 = Q$  and apply $(\ref{item:hyperbolic-descent})$ with $\varepsilon = 1/2$ to get a Carleson square $Q_1 \subset Q_0$, $\ell(Q_1) = 2^{-N} \ell(Q_0)$ with $1 - |F(z_{Q_1})| \le 2^{-1} \bigl (1-|F(z_{Q_0})| \bigr )$. By induction, we obtain a nested collection $\{Q_n\}$ of Carleson squares with $\ell(Q_n) = 2^{-N} \ell(Q_{n-1})$ and $1 - |F(z_{Q_n})| \le 2^{-1} \bigl (1-|F(z_{Q_{n-1}})| \bigr )$ for $n = 1, 2, \dots$.

Consider the point $\xi \in \bigcap_n \overline{Q_n} \in \partial \mathbb{D}$. It is not difficult to see that the geodesic ray $\gamma(z) = [z, \xi)$ satisfies the hypotheses of (\ref{item:quasigeodesic-condition}), since it lies within a bounded distance of the union of the geodesic segments joining the centers of the Carleson squares $Q_n$ for $n = 1, 2, \dots$. 

\medskip

(\ref{item:quasigeodesic-condition}) $\Rightarrow$ (\ref{item:hyperbolic-derivative}). Let $z \in \mathbb{D}$ be a point in the unit disk and $\gamma(z)$ be the geodesic ray given by (\ref{item:quasigeodesic-condition}). By the definition of $\gamma$, the hyperbolic diameter
$$
\diam_h \bigl (F(B_h(z, 1 + C/s) \bigr ) \ge s,
$$
for any $z \in \mathbb{D}$. A compactness argument shows that the hyperbolic diameter of $\diam_h(F(B_h(z, 1))$ is also bounded below by a positive constant, independent of $z \in \mathbb{D}$.
\end{proof}

\subsection{H\"older condition}
\label{sec:holder-condition}

The proof of (\ref{item:hyperbolic-descent}) $\Rightarrow$ (\ref{item:quasigeodesic-condition}) above shows that $\gamma$ is a ``good'' geodesic in the sense set that there exist an $0 < s \le 1$ and $C_2 > 0$ such that
$$
|d_h(0, F(w_1)) - d_h(0,F(w_2))| \ge s \cdot d_h (w_1, w_2) - C_2, \qquad w_1, w_2 \in \gamma.
$$
If $\xi \in \partial \mathbb{D}$ denotes the endpoint of $\gamma$, this last condition implies that
$$
\frac{1 - |F(r\xi)|}{(1-r)^s} \lesssim \frac{1 - |F(t\xi)|}{(1-t)^s}, \qquad 0 < t < r < 1.
$$
In particular,
$$
\limsup_{r \to 1} \frac{1-|F(r\xi)|}{(1-r)^s} < \infty.
$$
For $s = 1$, good geodesic rays are closely related to angular derivatives, see \cite[Section 3.1]{APHA}.

Fix a $\rho > 0$ and let $K \ge 0$ be a smooth even function supported on $(-\rho, \rho)$ with $\int_{\mathbb{R}} K(t)dt = 1$. In the proof below, we will use the {\em truncated square function}
$$
\mathcal Q(re^{i\theta}) = \int_{\{ |z| \le r, \, |z-e^{i\theta}| \le \rho(1-|z|)\} } K \biggl ( \frac{\arg z - \theta}{\rho} \biggr ) \cdot D^2_h F(z) \, \frac{dA(z)}{(1-|z|^2)^2},
$$
which is a smooth analogue of the {\em hyperbolic area function}
$$
\int_{\{ |z| \le r, \, |z-e^{i\theta}| \le \rho(1-|z|)\} } D^2_h F(z) \, \frac{dA(z)}{(1-|z|^2)^2}.
$$

\begin{proof}[Proof of Corollary \ref{holder-behaviour}]
According to \cite[Theorem 1.10]{square-functions}, for any inner function $F$, one has
\begin{equation}
\label{eq:square-functions}
\limsup_{r \to 1} \frac{|d_h(0, F(r\xi)) - \mathcal Q(r\xi)|}{\sqrt{\mathcal Q(r\xi) \log \log \mathcal Q(r\xi)}} \le C, \qquad m \text{ a.e.~}\xi \in \partial \mathbb{D}.
\end{equation}
From the definitions, it is clear that $\mathcal Q(r\xi) \lesssim d_h(0,r)$ for any analytic self-map of the unit disk and 
$\mathcal Q(r\xi) \asymp d_h(0,r)$ if $F$ has bounded compression. A casual inspection of (\ref{eq:square-functions}) shows that there exists an $0 < s \le 1$ such that
$$
s \, \le \, \frac{d_h(0, F(r\xi))}{d_h(0, r)} \, \le \, 1, \qquad r \ge r_0(\xi), \qquad m \text{ a.e.~}\xi \in \partial \mathbb{D},
$$
and the corollary is proved.
\end{proof}

\section{Aleksandrov-Clark measures}

In this section, we show (\ref{item:hyperbolic-derivative}) $\Leftrightarrow$ (\ref{item:ac-measures}), thereby characterizing BC mappings in terms of their Aleksandrov-Clark measures. The proof uses a number of identities from \cite{AAN} which express the hyperbolic derivative of an analytic self-map of the unit disk in terms of its Aleksandrov-Clark measures. These identities have played an important role in the works \cite{APHA, BN}.

\begin{proof} (\ref{item:hyperbolic-derivative}) $\Rightarrow$ (\ref{item:ac-measures}). Let $u$ be the Poisson extension of $\mu$. A simple calculation shows that the hyperbolic derivative 
\begin{equation}
\label{eq:identity2}
2 \, D_h F(z) = (1-|z|^2) \frac{|\nabla u(z)|}{u(z)}, \qquad z \in \mathbb{D}.
\end{equation}
Fix an arc $I \subset \partial \mathbb{D}$. Let $Q = Q(I)$ be the corresponding Carleson square and $N > 0$ be a large integer to be determined later. Applying Green's identity in the domain
$$
Q_N = \bigl \{ r \xi \in Q \, : \,  2^{-N} \ell (Q) < 1 - r < \ell(Q) \bigr \},
$$
to the functions $\log u(z)$ and $\log(1/|z|)$, we get
\begin{equation}
\label{eq:green2}
- \int_{\partial Q_N \cap \{ 1 - |z|  = \ell(Q) \} } \log  u(z) \, \frac{ds}{|z|} + \int_{\partial Q_N \cap \{ 1 - |z| = 2^{-N} \ell(Q) \} } \log u(z) \, \frac{ds}{|z|}
\end{equation}
\begin{equation*}
= - \int_{\partial Q_N} \partial_{\textbf{n}} \log u(z) \cdot \log \frac{1}{|z|} \, ds + \int_{Q_N} \Delta \log u(z) \cdot \log \frac{1}{|z|} \, dA(z).
\end{equation*}
By Harnack's inequality,
$$
\int_{\partial Q_N} \partial_{\textbf{n}} \log u(z) \cdot \log \frac{1}{|z|} \, ds = O(\ell(Q)).
$$
Using $\Delta (\log u) = - |\nabla u|^2/|u|^2$ together with (\ref{eq:identity2}) and Condition (\ref{item:hyperbolic-derivative}), we obtain
$$
\int_{Q_N} \Delta \log u(z) \cdot \log \frac{1}{|z|} \, dA(z) \, \asymp \, - \int_{Q_N} D_h^2 F(z) \, \frac{dA(z)}{1-|z|} \, \lesssim \, -N \ell(Q).
$$
Hence, (\ref{eq:green2}) reduces to
$$
- \int_{\partial Q_N \cap \{ 1 - |z| = \ell(Q) \} } \log u \, d\theta + \int_{\partial Q_N \cap \{ 1 - |z| = \ 2^{-N} \ell(Q) \} } \log u \, d\theta  \lesssim -N \ell(Q).
$$
Consequently, there exists a Carleson square $Q' \subset Q$, $\ell(Q') = 2^{-N} \ell(Q)$ such that $- \log u(z_Q) + \log u(z_{Q'}) \lesssim -N$, that is, there exists an absolute constant $c>0$ such that $u(z_{Q'}) \le 2^{-cN} u(z_Q)$. If $I$ is a heavy arc for $\mu$, then $\frac{\mu(I)}{|I|} \ge \frac{u(z_Q)}{100}$ and
$$
\frac{\mu(I')}{|I'|} \, \lesssim \, u(z_{Q'}) \le 2^{-cN} u(z_Q) \, \le \, 100 \cdot 2^{-cN} \cdot \frac{\mu(I)}{|I|},
$$
where $I' \subset \partial \mathbb{D}$ is the arc such that $Q' = Q(I')$.

\medskip

(\ref{item:ac-measures}) $\Rightarrow$ (\ref{item:hyperbolic-derivative}). Let $u$ be the Poisson integral of $\mu$. Since
$$
2\, D_h F(z) = (1-|z|^2) \frac{|\nabla u(z)|}{u(z)}, \qquad z \in \mathbb{D},
$$ 
in order to verify (\ref{item:hyperbolic-derivative}), it suffices to show that there exist positive constants $R, c > 0$ such that for any point $z \in \mathbb{D}$, there exists a point $w \in \mathbb{D}$ with
$$
d_h(z, w) < R \qquad \text{and} \qquad (1-|w|^2) \frac{|\nabla u(w)|}{u(w)} > c.
$$
For $z \in \mathbb{D} \setminus \{0\}$, let $I(z) \subset \partial \mathbb{D}$ be the arc centered at $z/|z|$ of length $2(1-|z|)$. Given an arc $I \subset \partial \mathbb{D}$, let $C \cdot I$ denote the arc with the same midpoint, but with $C$ times the length. The proof rests on the following observation:

\medskip

\noindent {\bf Claim.} Let $I = I(z) \subset \partial \mathbb{D}$ be an arc on the unit circle. If $2 I(z)$ is not heavy for $\mu$, then 
$$
2 \, D_h F(z) \, = \, (1-|z|^2) \frac{|\nabla u(z)|}{u(z)} \, \gtrsim \, 1. 
$$

\noindent {\bf Proof of the claim.} Since
$$
\frac{1}{\pi} \int_{-2}^2 \frac{dx}{x^2+1} > 0.7,
$$
for $z$ sufficiently close to the unit circle, the harmonic measure of the arc $2I(z) \subset \partial \mathbb{D}$ as viewed from $z$ exceeds $2/3$. By the conformal invariance of the harmonic measure, the image of the complementary arc $\partial \mathbb{D} \setminus 2 I(z)$ under the M\"obius transformation
\begin{equation*}
\xi \to \frac { 1- \overline{z} \xi}{\xi - z} \, : \, (\mathbb{D}, z) \, \to \, (\widehat{\mathbb{C}} \setminus \overline{\mathbb{D}}, \infty)
\end{equation*}
is an arc on the unit circle of length less than $2\pi/3$. Consequently,
$$
\biggl | \int_{\partial \mathbb{D} \setminus 2 I(z)}  \frac{1-|z|^2}{|\xi-z|^2 \frac{\xi-z}{1-\overline{z}\xi}} \, d\mu(\xi) \biggl | 
\, \ge \,
\frac{1}{2} \int_{\partial \mathbb{D} \setminus 2I(z)}  \frac{1-|z|^2}{|\xi-z|^2} \, d\mu(\xi). 
$$
In addition, we have the estimate 
\begin{equation*}
\int_{2I(z)}  \frac{1-|z|^2}{|\xi-z|^2} \, d\mu(\xi) \leq 4 \cdot  \frac{\mu (2 I(z))}{|I(z)|}. 
\end{equation*}
Putting the above facts together, we obtain
\begin{align*}
(1-|z|^2) |\nabla u(z)| & = 2 \, \biggl | \int_{\partial \mathbb{D}} \frac{1-|z|^2}{|\xi-z|^2 \frac{\xi-z}{1-\overline{z}\xi}} \, d\mu(\xi) \biggr | \\
& \ge 2 \, \biggl | \int_{\partial \mathbb{D} \setminus 2 I(z)}  \frac{1-|z|^2}{|\xi-z|^2 \frac{\xi-z}{1-\overline{z}\xi}} \,  d\mu(\xi) \biggl | \, - \, 4 \cdot \frac{\mu(2I(z))}{|I(z)|} \\
& \ge \int_{\partial \mathbb{D} \setminus 2I(z)}  \frac{1-|z|^2}{|\xi-z|^2}  \, d\mu(\xi) - 4 \cdot \frac{\mu(2I(z))}{|I(z)|} \\
& \ge \int_{\partial \mathbb{D}}  \frac{1-|z|^2}{|\xi-z|^2} \, d\mu(\xi) - 8 \cdot \frac{\mu(2I(z))}{|I(z)|}.
\end{align*}
Since $2I(z)$ is not heavy for $\mu$, we have
$$
\frac{\mu (2I(z))}{|I(z)|} < \frac{2}{100} \, u(z_{2I}). 
$$
Applying Harnack's inequality $u(z_{2I}) \leq 3\, u(z)$, we deduce
$$  (1-|z|^2) |\nabla u(z)| \, \ge \, u(z) - \frac{16}{100} \, u(z_{2I})
\, \ge \, u(z) - \frac{48}{100} \, u(z)
\, \ge \,  \frac{1}{2}\, u(z). 
$$

The claim tells us that if $I=I(z)$ is light with respect to $\mu$, then $z$ is a bounded hyperbolic distance away from the point $z_{I/2}$ where the hyperbolic derivative $D_h F$ is bounded  below.

It remains to deal with the case when $I=I(z)$ is a heavy with respect to $\mu$. By the hypothesis, there is an arc $J \subset I(z)$, $|J| \ge \delta |I(z)|$ such that
$$
\frac{\mu(J)}{|J|} \le \varepsilon \, \frac{\mu(I(z))}{|I(z)|}.
$$
We again consider two cases:
\begin{enumerate}
\item If $J$ is light with respect to $\mu$, then the claim tells us that $z$ is a bounded hyperbolic distance away from the point $z_{J/2}$ where the hyperbolic derivative is bounded from below.

\item On the other hand, if $J$ is heavy for $\mu$, then
$$
u(z_J) \, \le \, 100 \, \frac{\mu(J)}{|J|} \, \le \, 100 \varepsilon \, \frac{\mu(I(z))}{|I(z)|} \, \lesssim \, 100 \varepsilon \cdot u(z).
$$
Assuming that $\varepsilon > 0$ is small enough so that $u(z_J) < u(z)/2$, there exists a point $w$ on the hyperbolic segment joining $z$ and $z_J$ such that
$$
2\, D_h F(w) \, = \, (1-|w|^2) \frac{|\nabla u(w)|}{u(w)} \, \gtrsim \, 1.
$$
\end{enumerate}
The proof is complete.
\end{proof}

\section{Zero sets}

In this section, we show (\ref{item:hyperbolic-diameter}) $\Rightarrow$ (\ref{item:zero-sets}) $\Rightarrow$ (\ref{item:hyperbolic-derivative}), thereby characterizing the zero sets of BC mappings. Blaschke products whose zeros satisfy the Carleson condition, see item (b) in Condition \eqref{item:zero-sets}, are known as {\em Carleson-Newman Blaschke products}\/. The following lemma is well-known:

\begin{lemma}
\label{carleson-newman}
Suppose $F$ is a Carleson-Newman Blaschke product with zeros $\{z_n \}$.  For any $\varepsilon > 0$, there exists a $\delta > 0$ such that 
\begin{equation}
\label{CN}
\{ z \in \mathbb{D}: |F(z)| < \delta \} \subset \bigcup B_{h} (z_n , \varepsilon).     
\end{equation}
\end{lemma}

\begin{proof}[Proof]
For interpolating sequences, \eqref{CN} is just \cite[Lemma 1]{kerr-lawson}. Any Carleson-Newman Blaschke product can be expressed as a finite product of interpolating Blaschke products: $F = F_1 F_2 \cdots F_k$.
If for a given $\varepsilon$, the inclusion \eqref{CN} holds for $F_i$ with constant $\delta_i$, then it also holds for $F$ with
$\delta = \delta_1 \delta_2 \dots \delta_k$.
\end{proof}

\begin{remark}
From the Carleson condition, it is not difficult to see that if $F$ is a Carleson-Newman Blaschke product, then every point $z \in \mathbb{D}$ lies within a bounded hyperbolic distance of a point $w \in \mathbb{D}$ with $d_h(w, F^{-1}(0)) > 2/3$. Moreover, $w$ can be chosen so that $Q(w) \subset Q(z)$. However, one should be slightly careful since the Carleson condition does not imply that every point $z \in \mathbb{D}$ is a bounded hyperbolic distance away from a point $w \in \mathbb{D}$ for which $|F(w)| > 2/3$. 
\end{remark}

\begin{lemma}
\label{qfactor}
Fix an $0 < \eta < 1$. Let $F$ be a Blaschke product. If $|F(z)| > \eta > 0$, then
\begin{equation}
\label{eq:qfactor2}
1 - |F(z)|^2 \asymp \sum_{F(z_n) = 0} \frac{(1-|z|^2)(1-|z_n|^2)}{|1 - \overline{z}_n z |^2}.
\end{equation}
\end{lemma}

\begin{proof}
When $F(z) = m(z) = e^{i\theta} \cdot \frac{z-a}{1-\overline{a}z}$ is a M\"obius transformation, (\ref{eq:qfactor2}) is an equality by the Schwarz lemma:
$$
\frac{1 - |m(z)|^2}{1-|z|^2} \, = \, |m'(z)| \, = \, \frac{1-|a|^2}{|1-\overline{a}z|^2}.
$$
In the general case, when $F = \prod m_n$ is a product of M\"obius transformations,
$$
\log |F(z)| = \sum \log |m_n(z)|.
$$
The assumption implies that $|m_n(z)| > \eta$  for any $n = 1, 2, \dots$. Since the functions $1-t^2$ and $\log(1/t)$ are comparable for $t \in (\eta,1]$, we have
$$
\frac{1 - |F(z)|^2}{1-|z|^2} \, \asymp \, \sum \frac{1 - |m_n(z)|^2}{1-|z|^2} \, = \, \sum_{F(z_n) = 0} \frac{1-|z_n|^2}{|1 - \overline{z}_n z |^2},
$$
as desired.
\end{proof}

The above lemma can also be derived from \cite[Theorem 3.5]{mashreghi}. Combining Lemmas \ref{carleson-newman} and \ref{qfactor}, we obtain the following result:

\begin{corollary}
\label{noucoro}
Let $F$ be a Carleson-Newman Blaschke product. Then,
\begin{equation}
\label{eq:C-N}
1 - |F(z)| \asymp \sum_{F(z_n) = 0} \frac{(1-|z|^2)(1-|z_n|^2)}{|1 - \overline{z}_n z |^2}, \qquad z \in \mathbb{D}.
\end{equation}
\end{corollary}

Finally, we will need a distortion estimate for the pseudohyperbolic distance:

\begin{lemma}
\label{pseudohyperbolic}
Let $\rho(z, w) = |z-w| / |1-\overline{z}w| $ denote the pseudohyperbolic distance on the unit disk. Then:
\begin{enumerate}
\item[{\em (i)}] For any two points $z, w \in \mathbb{D}$, we have:
$$
1 - \rho(z, w)^2 = \frac{(1 - |z|^2)(1 - |w|^2)}{|1 - \overline{w}z |^2}.
$$
  
\item[{\em (ii)}] For any three points $z, w_1, w_2 \in \mathbb{D}$, we have:
$$
1 - \rho(z, w_2)^2 \le \frac{1 + \rho(w_1, w_2)}{1 - \rho(w_1, w_2)} \cdot (1 - \rho(z, w_1)^2 ).
$$
\end{enumerate}
\end{lemma}

The first statement is a straightforward calculation,
while the second statement follows from the reverse triangle inequality for the pseudohyperbolic distance
$$
\rho(z, w_2) \ge \frac{\rho(z, w_1) - \rho(w_1, w_2)}{1 - \rho(z, w_1) \rho(w_1, w_2)},
$$
e.g.~see \cite[Lemma 1.4]{garnett}.

\begin{proof} (\ref{item:hyperbolic-diameter}) $\Rightarrow$ (\ref{item:zero-sets}) By Lemma \ref{island}, $F$ is Blaschke product. We first show that $\sigma$ is a Carleson measure, so that $F$ is a Carleson-Newman Blaschke product. Indeed, if $\sigma$ were not a Carleson measure, then there would exist a sequence of  Carleson squares $\{Q_n\}$ such that 
$$
\frac{\sigma(Q_n)}{\ell(Q_n)} \to \infty.
$$
However, this would force the hyperbolic diameters of $F(B_h(z_{Q_n}, 1))$ to shrink to zero, contradicting (\ref{item:hyperbolic-diameter}). 

Fix an $\varepsilon > 0$ and let $Q$ be a heavy Carleson square. Using the equivalence of Conditions (\ref{item:hyperbolic-diameter}) and (\ref{item:hyperbolic-descent}) established in Section \ref{sec:geodesic-rays}, we find a sub-square $Q' \subset Q$, $\ell(Q') = 2^{-N} \ell(Q)$ such that
\begin{equation}
\label{eq:zc1}
1 - |F(z_{Q'})| \le \varepsilon (1 - |F(z_Q)|).
\end{equation}
Corollary \ref{noucoro} gives 
$$
\sum_{F(z_n) = 0} \frac{(1-|z_{Q'}|^2)(1-|z_n|^2)}{|1-\overline{z_n}z_{Q'}|^2} \lesssim \varepsilon  \sum_{F(z_n) = 0}  \frac{(1-|z_{Q}|^2)(1-|z_n|^2)}{|1-\overline{z_n}z_{Q}|^2}.
$$
Recall that for any Carleson square $Q'$, we have
$$
\frac{\sigma(Q')}{\ell(Q')} \lesssim  \sum_{F(z_n) = 0} \frac{(1-|z_{Q'}|^2)(1-|z_n|^2)}{|1-\overline{z_n}z_{Q'}|^2}.
$$
As $Q$ is heavy for $\mu$,
$$
\frac{\sigma(Q)}{\ell(Q)}  \geq \frac{1}{100}  \sum_{F(z_n) = 0} \frac{(1-|z_{Q}|^2)(1-|z_n|^2)}{|1-\overline{z_n}z_{Q}|^2}.
$$
Putting the above estimates together, we obtain $\frac{\sigma(Q')}{\ell(Q')} \lesssim \varepsilon \cdot \frac{\sigma(Q)}{\ell(Q)}$.

\medskip

(\ref{item:zero-sets}) $\Rightarrow$ (\ref{item:hyperbolic-derivative}). In order to verify (\ref{item:hyperbolic-derivative}), we show that there exist positive constants $R, c > 0$ such that for any point $z \in \mathbb{D}$, there exists a point $w \in \mathbb{D}$ with
$$
d_h(z, w) < R \qquad \text{and} \qquad D_h F(w) > c.
$$
By the remark following Lemma \ref{carleson-newman}, we may assume that $d_h (z , F^{-1} (0) ) $ is large.

\medskip

\noindent {\bf Claim.} 
Suppose $z \in \mathbb{D}$ is a point in the unit disk such that $d_h(z, F^{-1}(0)) > d_h(0, 2/3)$. If $2Q(z)$ is not heavy for $\sigma$, then $D_h F(z) \gtrsim 1$.

\medskip

\noindent {\bf Proof of the claim.} Connect the two endpoints of $2I(z)$ by a hyperbolic geodesic $\gamma \subset \mathbb{D}$ and let $\Gamma \supset \mathbb{D} \setminus 2Q(z)$ be the hyperbolic half-plane bounded by $\gamma$ which contains the origin. The image of $\Gamma$ under the M\"obius transformation
\begin{equation*}
\xi \to \frac { 1- \overline{z} \xi}{\xi - z} \, : \, (\mathbb{D}, z) \, \to \, (\widehat{\mathbb{C}} \setminus \overline{\mathbb{D}}, \infty)
\end{equation*}
is a lune in the exterior of the unit disk. By the harmonic measure estimate in the previous section, this lune is contained in a sector with a vertex at the origin, symmetric with respect to the ray $\{ -r z/|z|: 0 < r < \infty \}$ and having an opening angle less than $2\pi/3$. Consequently, if $\{z_n \}$ are the zeros of $F$,  
$$
\biggl | \sum_{\substack{F(z_n) = 0 \\ z_n \notin 2Q(z)}} \frac{(1-|z|^2)(1-|z_n|^2)}{|1-\overline{z_n}z|^2 \frac{z_n - z}{1-z_n\overline{z}}} \biggr | \ge
\frac{1}{2} \sum_{\substack{F(z_n) = 0 \\ z_n \notin 2Q(z)}} \frac{(1-|z|^2)(1-|z_n|^2)}{|1-\overline{z_n}z|^2 }.
$$

Logarithmic differentiation shows:
\begin{align*}
(1-|z|^2) \frac{|F'(z)|}{|F(z)|} & =
 \biggl | \sum_{F(z_n) = 0} \frac{(1-|z|^2)(1-|z_n|^2)}{|1-\overline{z_n}z|^2 \frac{z_n - z}{1-z_n\overline{z}}} \biggr | \\ 
 & \ge
 \biggl | \sum_{\substack{F(z_n) = 0 \\ z_n \notin 2Q(z)}} \frac{(1-|z|^2)(1-|z_n|^2)}{|1-\overline{z_n}z|^2 \frac{z_n - z}{1-z_n\overline{z}}} \biggr | - 
 \biggl | \sum_{\substack{F(z_n) = 0 \\ z_n \in 2Q(z)}} \frac{(1-|z|^2)(1-|z_n|^2)}{|1-\overline{z_n}z|^2 \frac{z_n - z}{1-z_n\overline{z}}} \biggr |
 \\ 
  & \ge
 \frac{1}{2} \sum_{\substack{F(z_n) = 0 \\ z_n \notin 2Q(z)}} \frac{(1-|z|^2)(1-|z_n|^2)}{|1-\overline{z_n}z|^2 } - 
 \frac{3}{2} \sum_{\substack{F(z_n) = 0 \\ z_n \in 2Q(z)}} \frac{(1-|z|^2)(1-|z_n|^2)}{|1-\overline{z_n}z|^2}
  \\ 
  & \ge
 \frac{1}{2} \sum_{F(z_n) = 0} \frac{(1-|z|^2)(1-|z_n|^2)}{|1-\overline{z_n}z|^2 } - 
2 \sum_{\substack{F(z_n) = 0 \\ z_n \in 2Q(z)}} \frac{(1-|z|^2)(1-|z_n|^2)}{|1-\overline{z_n}z|^2}.
\end{align*}
As $|1- \overline{z_n} z| \geq 1-|z|$,
$$
\sum_{\substack{F(z_n) = 0 \\ z_n \in 2Q(z)}} \frac{(1-|z|^2)(1-|z_n|^2)}{|1-\overline{z_n}z|^2}
\le 
8 \cdot \frac{\sigma(2Q(z))}{\ell(2Q(z))}.
$$
Since $2Q(z)$ is not heavy for $\sigma$, we have
\begin{align*}
\frac{\sigma(2Q(z))}{\ell(2Q(z))} & < \frac{1}{100} \sum_{F(z_n)=0} \frac{(1-|z_{2Q}|^2)(1-|z_n|^2)}{|1-\overline{z_n}z_{2Q}|^2} \\
& < \frac{3}{100} \sum_{F(z_n)=0} \frac{(1-|z|^2)(1-|z_n|^2)}{|1-\overline{z_n}z|^2},
\end{align*}
by Lemma \ref{pseudohyperbolic}. Putting the above equations together, we get
\begin{equation}
\label{eq:strange-equation}
(1-|z|^2) \frac{|F'(z)|}{|F(z)|} \ge 
\biggl (\frac{1}{2} - 2 \cdot 8 \cdot \frac{3}{100} \biggr )  \sum_{F(z_n) = 0} \frac{(1-|z|^2)(1-|z_n|^2)}{|1-\overline{z_n}z|^2},
\end{equation}
which is comparable to $1-|F(z)|^2$ by Corollary \ref{noucoro}. Since $d_h(z, F^{-1}(0)) \ge d_h(0, 2/3)$, Lemma \ref{carleson-newman} implies that $|F(z)| > \delta$ for some positive constant $\delta$. It follows from the above equation that the hyperbolic derivative of $F$ at $z$ is bounded below. This proves the claim.

\medskip

As explained previously, we may assume that $d_h (z , F^{-1} (0)) > 2/3$. By the claim, we may also assume that $Q(z)$ is heavy for $\sigma$, i.e.~
$$
\frac{\sigma(Q(z))}{\ell(Q(z))} \, \ge \, \frac{1}{100} \int_{\mathbb{D}} \frac{1-|z|^2}{|1-\overline{z} w|^2} \, d\sigma(w) \, \asymp \, 1 - |F(z)|.
$$
The hypothesis yields a Carleson square $Q' \subset Q(z)$ with 
$$
\ell(Q') \ge \delta \cdot \ell(Q) \qquad \text{and} \qquad 
\frac{\sigma(Q')}{\ell(Q')} \le \varepsilon \cdot \frac{\sigma(Q(z))}{\ell(Q(z))}.
$$
Adjusting the point $z_{Q'}$ a bounded hyperbolic distance if necessary, we may assume that $d_h(z_{Q'}, F^{-1}(0)) \ge d_h(0, 2/3)$. Applying the claim a second time, we may further assume that $Q'$ is heavy for $\sigma$, that is, 
$$
\frac{\sigma(Q')}{\ell(Q')} \, \ge \, \frac{1}{100} \int_{\mathbb{D}} \frac{1-|z_{Q'}|^2}{|1-\overline{z_{Q'}} w|^2} \, d\sigma(w) \, \asymp \, 1- |F(z_{Q'})|.
$$
Putting the above estimates together, we obtain
$$
1 - |F(z_{Q'})| \, \lesssim \, \frac{\sigma(Q')}{\ell(Q')} \, \le \, \varepsilon \cdot \frac{\sigma(Q(z))}{\ell(Q(z))} \, \lesssim \, \varepsilon (1-|F(z)|).
$$
Consequently, $1 - |F(z_{Q'})| \le (1-|F(z)|)/2$ if $\varepsilon > 0$ is sufficiently small. As the hyperbolic distance between $z$ and $z_{Q'}$ is bounded, there exists a point $w$ on the hyperbolic segment joining $z$ and $z_{Q'}$ such that $D_h F(w) \gtrsim \varepsilon > 0$.
\end{proof}

\section{A criterion for maximality}
\label{sec:maximality-criterion}

Recall from Lemma \ref{island} that any BC mapping is an indestructible Blaschke product. In this section, we strengthen this result by showing that every BC mapping is, in fact, a maximal Blaschke product.

It is easily verified that if $F$ is an analytic self-map of the unit disk, then the function
$$
u = \log \frac{2|F'|}{1-|F|^2}
$$
is a solution of the {\em Gauss curvature equation} (GCE)
$$
\Delta u = e^{2u} + 2\pi \sum_{c \in C} \delta_c,
$$
where $C = \crit F$. Conversely, Liouville's theorem's \cite[Theorem 2.7]{kraus} states that any solution of $\GCE_C$ with integral singularities is of the form
$$
u = \log \frac{2|F'|}{1-|F|^2},
$$
where $F$ is an analytic self-map of the unit disk with critical set $C$. Moreover, the corresponding Liouville map $F$ is determined uniquely up to post-composition with an automorphism of the unit disk.

As shown in \cite{heins}, see also \cite[Theorem 2.21]{kraus}, given a discrete set $C$ in the unit disk, if $\GCE_C$ has at least one solution, then it possesses a maximal solution $u_C$ which dominates all other solutions pointwise (the construction involves Perron's method). In this case, according to \cite[Theorem 1.2]{kraus}, the Liouville map $F_C$ of $u_C$ is a {\em maximal Blaschke product} with critical set $C$.

Before proving Theorem \ref{bounded-gap-theorem}, we present two applications to maximal Blaschke products:

\begin{corollary}
\label{like-heins-but-different}
Let $F$ be an analytic self-map of the unit disk. Suppose that for some $0 < \delta < 1$, each connected component of the set
$
E_\delta = \{z \in \mathbb{D} : (\lambda_F/\lambda_{\mathbb{D}})(z) \le \delta \}
$
is compactly contained in the unit disk. Then, $F$ is a maximal Blaschke product.
\end{corollary}

While the above criterion resembles the one in \cite[Section 25]{heins}, it is expressed in terms of hyperbolic derivatives, while Heins' criterion involves looking at the pre-images of small disks under $F$.

\begin{proof}
Suppose $F_C$ is the maximal Blaschke product with critical set $C = \crit F$. We want to show that $F = F_C$ up to post-composition with a M\"obius transformation. As usual, we write $u_F = \log \lambda_F$ and $u_{F_C} = \log \lambda_{F_C}$. As $u_F$ and $u_{F_C}$ satisfy $\GCE_C$, we have $u_F \le u_{F_C}$ by the maximality of $F_C$. The difference
$$
h = u_{F_C} - u_F
$$
is a subharmonic function since it has positive Laplacian. Outside $E_\delta$, the difference $h$ is bounded since
$$
u_F \, \ge \, u_{\mathbb{D}} - \log \frac{1}{\delta} \, \ge \,
u_{F_C} - \log \frac{1}{\delta}
$$
is already close to maximal. By the maximum principle, $h$ is bounded on whole the unit disk. Applying Theorem \ref{like-heins-but-different}, we conclude that there exists an automorphism $\tau$ of the unit disk such that  $ F = \tau \circ F_C$.
\end{proof}

\begin{corollary}
\label{apham-maximal}
Any analytic self-map of the unit disk which satisfies Condition {\em (\ref{item:distortion-crit})} is a maximal Blaschke product.
\end{corollary}

\begin{proof}
Since the critical set of $F$ is quasi-separated by Condition (\ref{item:distortion-crit}), when $\varepsilon > 0$ is sufficiently small, each connected component of $\bigcup_{c \in \crit F} \overline{B_h(c, \varepsilon)}$ is compactly contained in the unit disk. Also by Condition (\ref{item:distortion-crit}), for $\delta$ corresponding to $\varepsilon$,
$$
E_\delta \, = \, \{z \in \mathbb{D} : (\lambda_F/\lambda_{\mathbb{D}})(z) \le \delta \} \, \subset \, \bigcup_{c \in \crit F} \overline{B_h(c, \varepsilon)}.
$$
Consequently, the connected components of $E_\delta$ are also compactly contained in the unit disk. The result now follows from Corollary \ref{like-heins-but-different}.
\end{proof}

To prove Theorem \ref{bounded-gap-theorem}, we use the following facts:
\begin{enumerate}
\item[I.] According to Kato’s inequality  \cite[Proposition 6.6]{ponce}, if $u \in L^1_{\loc}$ and $\Delta u \in L^1_{\loc}$, then $\Delta u^+ \ge \chi_{\{u > 0\}} \cdot \Delta u$ in the sense of distributions, where as usual, $u^+ = \max(u,0)$ denotes the positive part of $u$.

\item[II.] If $0 \le u(z) \le C$ is a bounded subharmonic function on the unit disk, not identically zero, then
$$
\liminf_{r \to 1}  \,  \bigr| \{ \theta \in [0,2\pi) : u(re^{i\theta})  > 0 \} \bigr|  > 0.
$$
(If the above property is violated, then the sub-mean value property forces $u$ to be zero.)

\item[III.] Since an inner function $I$ has unimodular radial limits almost everywhere on the unit circle, for any $\varepsilon > 0$,
$$
\bigl | \{ \theta \in [0,2\pi) : | I(re^{i\theta}) | > 1 - \varepsilon \} \bigr | \, \to \, 2\pi,
$$
as $r \to 1$.
\end{enumerate}

\begin{proof}[Proof of Theorem \ref{bounded-gap-theorem}]
From \cite[Theorem 1.9]{square-functions}, it follows that $J$ is an inner function. By interchanging the roles of $I$ and $J$, it suffices to show that
$$
\frac{|I'(z)|}{1 - |I(z)|^2} \le  \frac{|J'(z)|}{1 - |J(z)|^2}, \quad z \in \mathbb{D}.
$$
Equivalently, we may show that for any $0 < \rho < 1$, the function
\begin{align*}
h_\rho(z) & = \biggl ( \log \frac{2\rho \, |I'(z)|}{1 - \rho^2 |I(z)|^2} - \log \frac{2 |J'(z)|}{1 - |J(z)|^2} \biggr )^+, \qquad z \in \mathbb{D},
\end{align*}
is identically zero. Evidently, the functions $\{ h_\rho(z) \}_{0 < \rho < 1}$ are non-negative. By the Schwarz lemma and the hypothesis of the theorem, they are uniformly bounded on the unit disk. Since their Laplacians are non-negative (by Fact I), each function $h_\rho(z)$ is subharmonic.

For the sake of contradiction, suppose that $h_\rho(z)$ is not identically zero for some $0 < \rho < 1$. 
By Facts II and III above, for any $\varepsilon > 0$,
$$
\liminf_{r \to 1} \ \bigl | \bigl  \{ \theta \in [0,2\pi) \, : \,  h_\rho(re^{i\theta}) > 0, \, | I(re^{i\theta}) | > 1 - \varepsilon \bigr \} \bigr | \, > \, 0.
$$
Evidently, for $\rho < \sigma < 1$, we have
$$
h_\sigma(z) - h_\rho(z) \ge \biggl ( \log \frac{2\sigma}{1-\sigma^2 |I(z)|^2} - \log \frac{2\rho}{1-\rho^2 |I(z)|^2} \biggr ) \, \chi_{\{ h_\rho > 0 \}} (z).
$$
When $|I(z)|$ is close to 1, with $\rho$ fixed and $\sigma$ close to 1, the expression in parentheses becomes arbitrary large. This contradicts the uniform boundedness of $\{ h_\sigma(z) \}_{0 < \sigma < 1}$. The proof is complete.
\end{proof}

\section{Critical sets}

In this section, we describe the possible critical sets of BC mappings by showing the equivalence of Conditions (\ref{item:distortion-crit}) and (\ref{item:critical-points}).

\begin{lemma}
\label{cs-jensen-bound}
Let $F$ be an analytic self-map of the unit disk with $F(0) = 0$ and $F'(0) \ne 0$. If the critical set of $F$ is quasi-separated, then
\begin{equation}
\label{eq:cs-jensen-bound}
\sum_{\crit F \, \cap \, B(0, r)} \log \frac{1}{|c|} \, \le \,  \frac{1}{2\pi} \int_{|z| < r} \bigl (1 - D_h^2 F(z) \bigr ) \, \frac{dA(z)}{1-|z|} + \log \frac{1}{|F'(0)|} + O(1).
\end{equation}
\end{lemma}

\begin{proof}
Plugging in the identity
$$
\Delta \log \frac{1-|F(z)|^2}{1-|z|^2} = 4 \, \biggl ( \frac{1}{(1-|z|^2)^2} -  \frac{|F' (z)|^2}{(1-|F (z)|^2)^2} \biggr )
$$
into Jensen's formula
$$
\frac{1}{2\pi} \int_0^{2 \pi} u(re^{i\theta}) d\theta = u(0) + 
\frac{1}{2\pi} \int_{|z|<r} \Delta u(z) \log \frac{r}{|z|} \, dA (z),
$$
we obtain
\begin{equation*}
\label{eq:UD-jensen1}
\frac{1}{2\pi} \int_0^{2 \pi} \log  \frac{1-|F(re^{i\theta})|^2}{1- r^2} \, d\theta = \frac{1}{2\pi} \int_{|z| < r} \bigl (1 - D_h^2 F(z) \bigr ) \frac{dA(z)}{1-|z|} + O(1).
\end{equation*}
Here is a detailed justification of the $O(1)$ error term in the equation above:
\begin{itemize}
\item Since $\log (r/|z|) - \log (1/|z|) = \log r \asymp (1-r)$ and $\Delta u(z) \le \frac{4}{(1-|z|^2)^2}$,
$$
\biggl | \int_{|z|<r} \Delta u(z) \log \frac{r}{|z|} \, dA (z) \, - \,  \int_{|z|<r} \Delta u(z) \log \frac{1}{|z|} \, dA (z) \biggr |
$$
is $\lesssim (1-r) A_h(B(0,r)) = O(1)$.

\item Similarly, since $|\log(1/|z|) - (1-|z|)| \lesssim (1-|z|)^2$ and $\Delta u(z) \le \frac{4}{(1-|z|^2)^2}$, we have
$$
\biggl | \int_{|z|<r} \Delta u(z) \log \frac{1}{|z|} \, dA (z) \, - \, \int_{|z|<r} \Delta u(z) (1-|z|) \, dA (z) \biggr | \, = \, O(1).
$$
\item Finally, from 
\begin{align*}
\biggl | \frac{4}{(1-|z|^2)^2} - \frac{1}{(1-|z|)^2} \biggr | & \le \biggl | \frac{2}{1-|z|^2} + \frac{1}{1-|z|} \biggr | \cdot
\biggl | \frac{2}{1-|z|^2} - \frac{1}{1-|z|} \biggr | \\ & \lesssim \,  1/ (1-|z|)
\end{align*}
and $\Delta u(z) = \frac{4}{(1-|z|^2)^2} \cdot (1 - D_h^2 F(z) )$, we get
$$
\biggl | \int_{|z|<r} \Delta u(z) (1-|z|) \, dA (z) \, - \, \int_{|z|<r} (1 - D_h^2 F (z) ) \, \frac{dA (z)}{1-|z|} \biggr | \, = \, O(1).
$$
\end{itemize}
On the other hand,
\begin{equation*}
\label{eq:UD-jensen2}
\frac{1}{2\pi} \int_0^{2 \pi} \log |F'(re^{i\theta})| \, d\theta = \log |F'(0)| + \sum_{\crit F \, \cap \, B(0, r)} \log \frac{r}{|c|}.
\end{equation*}
By the Schwarz lemma,
\begin{equation}
\label{eq:UD-jensen3}
\sum_{\crit F \, \cap \, B(0, r)} \log \frac{r}{|c|} \, \le \,  \frac{1}{2\pi} \int_{|z| < r} \bigl (1 - D_h^2 F(z) \bigr ) \, \frac{dA (z)}{1-|z|}
+ \log \frac{1}{|F'(0)|} + O(1).
\end{equation}

It remains to replace $\log(r/|c|)$ with $\log(1/|c|)$ on the left hand side of (\ref{eq:UD-jensen3}). To do so, we examine the difference
\begin{equation}
\label{eq:UD-jensen4}
\sum_{\crit F \, \cap \, B(0, r)} \biggl ( \log \frac{1}{|c|} - \log \frac{r}{|c|} \biggr ) \, \asymp \, (1-r) \cdot \#\{ c \in B(0,r) : F'(c) = 0 \}.
\end{equation}
Since the critical set of $F$ is quasi-separated, $F$ has at most $O(1/(1-r))$ critical points in $B(0, r)$ and so the right hand side of (\ref{eq:UD-jensen4}) is $O(1)$. Adding this estimate to (\ref{eq:UD-jensen3}) yields the desired bound (\ref{eq:cs-jensen-bound}).
\end{proof}

\begin{remark}
Without the quasi-separation assumption, the number of critical points of $F$ inside $B(0,r)$ can be significantly larger. Indeed, using \cite[Theorem 6]{shapiro-shields}, one can construct an analytic self-map of the disk such that
$$
\limsup_{r \to 1} \frac { \#\{ c \in B(0,r) : F'(c) = 0 \} }{\frac{1}{1-r} \log \frac{1}{1-r}} > 0,
$$
in which case, (\ref{eq:cs-jensen-bound}) fails even with an $o(\log \frac{1}{1-r})$ error term.
\end{remark}

\begin{proof} (\ref{item:distortion-crit}) $\Rightarrow$ (\ref{item:critical-points}).
Since Condition (\ref{item:distortion-crit}) already assumes that the critical set $C = \crit(F)$ is quasi-separated and Corollary \ref{apham-maximal} ensures that $F$ is a maximal Blaschke product, it remains to verify that $D^+(C) < 1$.

A compactness argument shows that (\ref{item:distortion-crit}) implies a stronger form of Condition (\ref{item:hyperbolic-derivative}): there exist constants $c, \delta > 0$ such that every ball of hyperbolic radius 1 contains a sub-ball of radius $\delta$ on which the hyperbolic derivative $D_h F (z) > c$. Consequently, for some $\beta > 0$, 
$$
\int_{|z| < r} D_h^2 F(z) \, \frac{dA(z)}{1-|z|^2} \ge \beta \int_{|z| < r} \frac{dA(z)}{1-|z|^2} - O(1).
$$
From Lemma \ref{cs-jensen-bound}, it follows that $D^+(C) \le 1-\beta$. 

\medskip

(\ref{item:critical-points}) $\Rightarrow$ (\ref{item:distortion-crit}).
Let $C$ be a quasi-separated subset of the unit disk with $D^+(C) < 1$ and $F = F_C$ be a maximal Blaschke product with critical set $C$ (the existence of $F_C$ follows from Remark 9 in the introduction, as well as from the argument below). Assume that the hyperbolic distance $d_h(C, p) > \delta$. We want to show that the hyperbolic derivative
$D_h F(p) = (\lambda_{F}/\lambda_{\mathbb{D}})(p) > \varepsilon$.

By Möbius invariance, we may assume that $p = 0$ and $F(0) = 0$. With this normalization, our objective amounts to showing that $\lambda_F(0) = 2|F'(0)|$ is bounded below.

Choose $a$ so that $D^+(C) < a < 1$. By \cite[Lemma 4.1]{seip}, one can construct an analytic function $f(z)$ on the unit disk with
\begin{itemize}
\item $|f(z)| < 1/(1-|z|)^a$,
 \item $f$ vanishes on $C$,
\item $|f(0)| > \varepsilon_0 = \varepsilon_0(C, a) > 0$ is bounded from below by a positive constant.
\end{itemize}
Set $I_a := \int_0^1 \frac{dr}{(1-r)^a} < \infty$. It is easy to see that $$h(z) = (1/I_a) \int_0^z f(t) dt$$ is a function in the $H^\infty$ unit ball with $h(0) = 0$ and critical set $C$, while $|h'(0)| \ge \varepsilon_0/I_a$.
Since $F$ is a maximal Blaschke product,  $ \lambda_{F}(0) \ge \lambda_h(0) \ge 2 \varepsilon_0/I_a$ is bounded below as desired.
\end{proof}

{\small {\bf Acknowledgements.}
This research was supported by the Israel Science Foundation (grant 3134/21), the Generalitat de Catalunya (grant 2021 SGR 00071), the Spanish Ministerio de Ciencia e Innovaci\'on (project PID2021-123151NB-I00) and the Spanish Research Agency through the Mar\'ia de Maeztu Program (CEX2020-001084-M).}

\vspace{-4pt}

\bibliographystyle{amsplain}

\end{document}